\documentclass[12pt]{amsart}
\usepackage{amssymb,mathrsfs,longtable,textcomp}
\usepackage[matrix,arrow,curve]{xy}
\usepackage{setspace}
\usepackage{multirow}
\usepackage{epsfig,color}
\usepackage{url}

\newcounter{cequation}[section]

\newtheorem{theorem}[cequation]{Theorem}
\newtheorem*{theorem*}{Theorem}
\newtheorem{lemma}[cequation]{Lemma}
\newtheorem{corollary}[cequation]{Corollary}

\newtheorem{proposition}[cequation]{Proposition}

\theoremstyle{definition}

\newtheorem{definition}[cequation]{Definition}
\newtheorem*{definition*}{Definition}

\theoremstyle{remark}
\newtheorem{remark}[cequation]{Remark}

\makeatletter\@addtoreset{equation}{section}
\makeatletter\@addtoreset{section}{part}

\makeatother

\def \O {\mathcal{O}}

\def \CC {\Bbbk}

\def \PP {\mathcal{P}}

\def \P {\mathbb{P}}
\def \PP {\mathbb{P}}

\def \ZZ {\mathbb{Z}}

\def \A {\mathbb{A}}

\def \Aff {\mathbb{A}}

\newcommand{\Proj}{\mathrm{Proj}\,}

\def \Aut {\mathrm{Aut}}
\def \Cl {\mathrm{Cl}\,}
\def \Pic {\mathrm{Pic}\,}

\newcommand{\Spec}{\mathrm{Spec}}

\def \ge {\geqslant}
\def \le {\leqslant}

\title{Automorphisms of weighted complete intersections}

\author{Victor Przyjalkowski and Constantin Shramov}

\address{\emph{Victor Przyjalkowski}
\newline
\textnormal{Steklov Mathematical Institute of RAS, 8 Gubkina street, Moscow 119991, Russia.
}
\newline
\textnormal{HSE University,
Russian Federation,
Laboratory of Mirror Symmetry, NRU HSE, 6 Usacheva str., Moscow, Russia, 119048.
}
\newline
\textnormal{\texttt{victorprz@mi.ras.ru, victorprz@gmail.com}}}

\address{\emph{Constantin Shramov}
\newline
\textnormal{Steklov Mathematical Institute of RAS,
8 Gubkina street, Moscow 119991, Russia.
}
\newline
\textnormal{
HSE University, Russian Federation,
Laboratory of Algebraic Geometry, 6 Usacheva str., Moscow, 119048, Russia.
}
\newline
\textnormal{\texttt{costya.shramov@gmail.com}}}

\thanks{Victor Przyjalkowski was partially supported by Laboratory of Mirror Symmetry NRU HSE, RF Government grant, ag. \textnumero~14.641.31.0001.
Constantin Shramov was supported by the Russian Academic Excellence Project ``5-100'' and by the Foundation for the
Advancement of Theoretical Physics and Mathematics ``BASIS''.
Both authors are Young Russian Mathematics award winners and would like to thank
its sponsors and jury.
}

\begin{document}

\begin{abstract}
We show that smooth well formed weighted complete intersections have finite
automorphism groups, with several obvious exceptions.
\end{abstract}

\maketitle

\section{Introduction}
\label{section:intro}

Studying algebraic varieties, it is important to understand their automorphism groups.
In some particular cases these groups have especially nice structure. For instance,
recall the following classical result due to H.\,Matsumura and P.\,Monsky (cf.~\cite[Lemma 14.2]{KS58}).

\begin{theorem}[{see~\cite[Theorems~1 and~2]{MM63}}]
\label{theorem:MM}
Let $X\subset\P^N$, $N\ge 3$, be a smooth hypersurface of degree  $d\ge 3$.
Suppose that $(N,d)\neq (3,4)$.
Then the group~\mbox{$\Aut(X)$} is finite.
\end{theorem}

The following beautiful generalization of Theorem~\ref{theorem:MM} was proved
by O.\,Benoist.

\begin{theorem}[{\cite[Theorem~3.1]{Be13}}]
\label{theorem:Benoist}
Let $X$ be a smooth
complete intersection of dimension at least $2$ in
$\P^N$ that is not contained in a hyperlplane.
Suppose that $X$ does not coincide with~$\P^N$, is not a
quadric hypersurface in $\P^N$, and is not a $K3$ surface.
Then the group~\mbox{$\Aut(X)$} is finite.
\end{theorem}

The goal of this paper is to generalize Theorems~\ref{theorem:MM}
and~\ref{theorem:Benoist} to the case of smooth weighted complete intersections.
We refer the reader to~\cite{Do82} and~\cite{IF00} (or to~\S\ref{section:preliminaries} below) for
definitions and basic properties of weighted projective spaces and complete intersections therein.
Our main result is as follows.

\begin{theorem}
\label{theorem:automorphisms}
Let $X$ be a smooth well formed
weighted complete intersection of dimension~$n$.
Suppose that either $n\ge 3$, or $K_X\neq 0$.
Then the group~\mbox{$\Aut(X)$} is finite unless
$X$ is isomorphic either to $\P^n$ or to a quadric hypersurface in~$\P^{n+1}$.
\end{theorem}

Under a minor additional assumption (cf. Definition~\ref{definition:cone}
below) one can make the assertion of
Theorem~\ref{theorem:automorphisms} more precise.

\begin{corollary}
\label{corollary:automorphisms}
Let $X\subset\P$ be a smooth well formed
weighted complete intersection of dimension $n$ that
is not an intersection with a linear cone.
Suppose that either $n\ge 3$, or~\mbox{$K_X\neq 0$}.
Then the group~\mbox{$\Aut(X)$} is finite unless
$X=\P\cong\P^n$ or $X$ is a quadric hypersurface in~\mbox{$\P\cong\P^{n+1}$}.
\end{corollary}

Note that if $X$ is not an intersection with a linear cone,
then the assumption of Theorem~\ref{theorem:automorphisms}
is equivalent to the requirement
that $X$ is not one of the weighted complete intersections
listed in Table~\ref{table:K3} below.

We refer the reader to \cite{HMX}, \cite[Theorem~1.1.2]{KPS18}, and \cite[Theorem~1.2]{CPS19} for other results concerning finiteness
of automorphism groups.

Theorem~\ref{theorem:automorphisms} is mostly implied by the results of
\cite{Fle81} (see Theorem~\ref{theorem:Flenner} below). However, some cases are not
covered by \cite{Fle81} and have to be classified and treated separately
(see Proposition~\ref{proposition:low-coindex}(iv)
and Lemma~\ref{lemma:low-coindex}).

To deduce Corollary~\ref{corollary:automorphisms}
from Theorem~\ref{theorem:automorphisms}, we
need the the following assertion that is well known to
experts but for which we did not manage to find a proper reference (and which we find interesting on its own).
We will say that a weighted complete intersection~\mbox{$X\subset\P=\P(a_0,\ldots,a_N)$}
of multidegree~\mbox{$(d_1,\ldots,d_k)$} is \emph{normalized} if the inequalities~\mbox{$a_0\le\ldots\le a_N$} and $d_1\le\ldots\le d_k$ hold.

\begin{proposition}[{cf.~\cite[Lemma~18.3]{IF00}}]
\label{proposition:unique-embedding}
Let $X\subset\P(a_0,\ldots,a_N)$
and~\mbox{$X'\subset\P(a'_0,\ldots,a'_{N'})$}
be normalized quasi-smooth well formed weighted complete intersections
of multidegrees $(d_1,\ldots,d_k)$ and $(d'_1,\ldots,d'_{k'})$, respectively, such that $X$ and~$X'$ are
not intersections with linear cones. Suppose that $X\cong X'$ and $\dim X\ge 3$.
Then~\mbox{$N=N'$}, $k=k'$, $a_i=a'_i$ for every $0\le i\le N$, and $d_j=d'_j$ for every $1\le j\le k$.
\end{proposition}

When the first draft of this paper was completed,
A.\,Massarenti informed us that a result essentially similar
to Theorem~\ref{theorem:automorphisms} was proved earlier
in \cite[Proposition~5.7]{ACM}.
Note however that in \cite[\S5]{ACM} the authors work
with smooth weighted complete intersections subject to certain
strong additional assumptions (see \cite[Assumptions~5.2]{ACM} for details).

Throughout the paper we work over an algebraically closed field $\CC$ of characteristic zero.

\smallskip
The plan of our paper is as follows.
In~\S\ref{section:preliminaries} we recall some preliminary facts on weighted complete intersections.
In~\S\ref{section:linear-normality} we show uniqueness of presentation of a variety as a weighted complete intersection,
that is, we prove Proposition~\ref{proposition:unique-embedding}. Finally, in~\S\ref{section:automorphisms} we  prove Theorem~\ref{theorem:automorphisms}
and Corollary~\ref{corollary:automorphisms}.

\smallskip
We are grateful to
B.\,Fu, B. van~Geemin, A.\,Kuznetsov, A.\,Massarenti,
and T.\,Sano for useful discussions. 
We also thank the referee for his suggestions regarding the first draft of the
paper.

\section{Preliminaries}
\label{section:preliminaries}

In this section we recall the basic properties of weighted complete intersections. We refer the reader
to~\cite{Do82} and~\cite{IF00} for more details. Some properties of smooth
weighted complete intersections can be also found in the earlier paper~\cite{Mo75}.

Let $a_0,\ldots,a_N$ be positive integers. Consider the graded algebra~\mbox{$\CC[x_0,\ldots,x_N]$},
where the grading is defined by assigning the weights $a_i$ to the variables~$x_i$.
Put
$$
\P=\P(a_0,\ldots,a_N)=\mathrm{Proj}\,\CC[x_0,\ldots,x_N].
$$

The weighted projective space $\P$ is said to be \emph{well formed} if the greatest common divisor of any $N$ of the weights~$a_i$ is~$1$. Every weighted projective space is isomorphic to a well formed one, see~\cite[1.3.1]{Do82}.
A subvariety $X\subset \P$ is said to be \emph{well formed}
if~$\P$ is well formed and
$$
\mathrm{codim}_X \left( X\cap\mathrm{Sing}\,\P \right)\ge 2,
$$
where the dimension of the empty set is defined to be $-1$.

We say that a subvariety $X\subset\P$ of codimension $k$ is a \emph{weighted complete
intersection of multidegree $(d_1,\ldots,d_k)$} if its weighted homogeneous ideal in $\CC[x_0,\ldots,x_N]$
is generated by a regular sequence of $k$ homogeneous elements of degrees $d_1,\ldots,d_k$.
This is equivalent to
the requirement that the codimension of (every irreducible component of)
the variety~$X$
equals the (minimal)
number of generators of the weighted homogeneous ideal of~$X$,
cf.~\cite[Theorem~II.8.21A(c)]{Ha77}.
Note that $\P$ can be thought of as a weighted complete
intersection of codimension $0$ in itself; this gives us a smooth Fano variety if and only if~\mbox{$\P\cong\P^N$}.

\begin{definition}[{see~\cite[Definition 6.3]{IF00}}]
\label{definition: quasi-smoothness}
Let $p\colon \mathbb A^{N+1}\setminus \{0\}\to \P$ be the natural projection to the weighted projective space. A subvariety $X\subset \P$
is called \emph{quasi-smooth} if $p^{-1}(X)$ is smooth.
\end{definition}

Note that a smooth well formed weighted complete intersection is always quasi-smooth,
see~\cite[Corollary~2.14]{PrzyalkowskiShramov-Weighted}.

The following definition describes weighted complete intersections
that are to a certain extent analogous to complete intersections
in a usual projective space that are contained in a hyperplane.

\begin{definition}[{cf. \cite[Definition~6.5]{IF00}}]
\label{definition:cone}
A weighted complete intersection~$X\subset\P$
is said to be \emph{an intersection
with a linear cone} if one has $d_j=a_i$ for some~$i$ and~$j$.
\end{definition}

\begin{remark}
\label{remark:cone}
A general quasi-smooth well formed weighted complete intersection
is isomorphic to a quasi-smooth well formed weighted complete intersection
that is not an intersection with a linear cone,
cf. \cite[Remark~5.2]{PrzyalkowskiShramov-Weighted}.
Note however that this does not hold without the
generality assumption.
For instance, a general weighted complete intersection
of bidegree~\mbox{$(2,4)$} in $\P(1^{n+2},2)$ is isomorphic to
a quartic hypersurface in $\P^{n+1}$, while certain
weighted complete intersections of this type are isomorphic to
double covers of an $n$-dimensional quadric branched over
an intersection with a quartic.
\end{remark}

Given a subvariety $X\subset\P$,
we denote by~$\O_X(1)$ the restriction of the
sheaf~\mbox{$\O_{\P}(1)$} to~$X$, see~\mbox{\cite[1.4.1]{Do82}}.
Note that the sheaf~\mbox{$\O_{\P}(1)$} may be not invertible.
However,
if $X$ is well formed, then $\O_X(1)$ is a well-defined divisorial sheaf on $X$.
Furthermore, if $X$ is well formed and smooth,
then $\O_X(1)$ is a line bundle on~$X$.

\begin{lemma}[{\cite[Remark~4.2]{Okada2}, \cite[Proposition~2.3]{PST}, cf. \cite[Theorem 3.7]{Mo75}}]
\label{lemma:Okada}
Let~$X$ be a quasi-smooth
well formed weighted complete intersection of dimension at least~$3$.
Then the class of the divisorial sheaf $\mathcal{O}_{X}(1)$ generates
the group $\Cl(X)$ of classes of Weil divisors on~$X$. In particular, under the additional assumption that $X$ is smooth,
the class of the line bundle $\mathcal{O}_{X}(1)$ generates
the group~$\Pic(X)$.
\end{lemma}

One can describe the canonical class of a weighted complete intersection.
For a weighted complete intersection $X$ of multidegree $(d_1,\ldots,d_k)$
in~$\P$, define
\begin{equation*}\label{eq:i-x}
i_X=\sum a_j-\sum d_i.
\end{equation*}
Let $\omega_X$ be the dualizing sheaf on $X$.

\begin{theorem}[{see~\cite[Theorem 3.3.4]{Do82}, \cite[6.14]{IF00}}]
\label{theorem:adjunction}
Let $X$ be a quasi-smooth
well formed weighted complete intersection.
Then
$$
\omega_X=\O_X\left(-i_X\right).
$$
\end{theorem}

Using the bounds on numerical invariants of smooth weighted complete intersections
found in \cite[Theorem~1.3]{ChenChenChen}, \cite[Theorem~1.1]{PrzyalkowskiShramov-Weighted}, and~\cite[Corollary~5.3(i)]{PST},
one can easily obtain the classically
known lists of all smooth Fano weighted complete intersections of small dimensions. Namely, we
have the following.

\begin{lemma}
\label{lemma:small-dimension}
Let $X$ be a smooth well formed Fano weighted complete intersection of dimension at most $2$
in $\P$ that is not an intersection with a linear cone.
Then $X$ is one of the varieties listed
in Table~\ref{table:dim 12}.
\end{lemma}

\begin{center}
\begin{longtable}{||c|c|c|}
\hline
No. & $\P$ & Degrees \\
\hline
\hline
\multicolumn{3}{|c|}{dimension $1$}\\
\hline
\hline
1.1 & $\P^2$ & $2$ \\
\hline
1.2 & $\P^1$ & $\varnothing$ \\
\hline
\hline
\multicolumn{3}{|c|}{dimension $2$}\\
\hline
\hline
2.1 & $\P(1^2,2,3)$ & $6$  \\
\hline
2.2 & $\P(1^3,2)$ & $4$  \\
\hline
2.3 & $\P^3$ & $3$ \\
\hline
2.4 & $\P^4$ & $2,2$ \\
\hline
2.5 & $\P^3$ & $2$ \\
\hline
2.6 & $\P^2$ & $\varnothing$ \\
\hline
\caption[]{Fano weighted complete intersections in dimensions
$1$ and $2$}\label{table:dim 12}
\end{longtable}
\end{center}

\begin{remark}\label{remark:dP}
Let $X$ be a smooth well formed Fano weighted complete intersection of dimension $2$.
If we do not assume that $X$ is not an intersection with a linear cone, we cannot
use the classification provided by Lemma~\ref{lemma:small-dimension}.
However, Lemma~\ref{lemma:small-dimension} applied together with Remark~\ref{remark:cone}
shows that if $i_X=1$, then $X$ is a del Pezzo surface of (anticanonical) degree at
most $4$.
\end{remark}

Recall that the \emph{Fano index} of a Fano variety $X$ is defined as the maximal integer $m$ such that the canonical
class $K_X$ is divisible by $m$ in the Picard group of $X$.
Theorem~\ref{theorem:adjunction}
and Lemmas~\ref{lemma:Okada} and~\ref{lemma:small-dimension}
imply the following.

\begin{corollary}\label{corollary:index}
Let $X$ be a smooth Fano
well formed weighted complete intersection of dimension at least $2$.
Then the Fano index of~$X$ equals~$i_X$.
\end{corollary}

\begin{proof}
Suppose that $\dim X=2$.
Note that the Fano index is constant in the family of smooth weighted complete intersections
of a given multidegree in a given weighted projective space.
Similarly, $i_X$ is constant in such a family.
Thus, by Remark~\ref{remark:cone}
we may assume that~$X$ is not an intersection with a linear cone.
Now the assertion follows from the classification
provided in Lemma~\ref{lemma:small-dimension}.

If $\dim X\ge 3$, then we apply
Theorem~\ref{theorem:adjunction} together with Lemma~\ref{lemma:Okada}.
\end{proof}

Note that the assertion of Corollary~\ref{corollary:index} fails in dimension $1$:
if $X$ is a conic in $\P^2$, then~\mbox{$i_X=1$}, while the Fano index of $X$ equals~$2$.

\begin{lemma}
\label{lemma:extension}
Let $X\subset\P(a_0,\ldots,a_N)$ be a smooth well formed weighted complete
intersection of multidegree $(d_1,\ldots,d_k)$. Then
a general weighted complete intersection $X'$
of multidegree $(d_1,\ldots,d_k)$ in
$\P(1,a_0,\ldots,a_N)$ is smooth and well formed,
and $i_{X'}=i_X+1$.
\end{lemma}
\begin{proof}
Straightforward.
\end{proof}

For the converse of Lemma~\ref{lemma:extension}, see \cite[Theorem~1.2]{PST}.

Similarly to Lemma~\ref{lemma:small-dimension}, we can classify
three-dimensional smooth well formed Fano weighted complete intersections that
are not intersections with a linear cone
(see~\cite[Table 2]{PSh18}). This together with
Lemma~\ref{lemma:extension} allows us to
classify smooth well formed weighted complete intersections of
dimension up to $2$ with trivial canonical class.

\begin{lemma}
\label{lemma:K3}
Let $X$ be a smooth well formed
weighted complete intersection of dimension at most $2$
in $\P$ that is not an intersection with a linear cone. Suppose that $K_X=0$.
Then~$X$ is one of the varieties listed
in Table~\ref{table:K3}.
\end{lemma}

\begin{center}
\begin{longtable}{||c|c|c|}
\hline
No. & $\P$ & Degrees \\
\hline
\hline
\multicolumn{3}{|c|}{dimension $1$}\\
\hline
\hline
1.1 & $\P(1,2,3)$ & $6$ \\
\hline
1.2 & $\P(1,1,2)$ & $4$ \\
\hline
1.3 & $\P^2$ & $3$ \\
\hline
1.4 & $\P^3$ & $2,2$ \\
\hline
\hline
\multicolumn{3}{|c|}{dimension $2$}\\
\hline
\hline
2.1 & $\P(1^3,3)$ & $6$  \\
\hline
2.2 & $\P^3$ & $4$  \\
\hline
2.3 & $\P^4$ & $2,3$ \\
\hline
2.4 & $\P^5$ & $2,2,2$ \\
\hline
\caption[]{Calabi--Yau weighted complete intersections in dimensions $1$ and $2$}\label{table:K3}
\end{longtable}
\end{center}

\begin{remark}
\label{remark:families}
Note that each of the four families of elliptic curves listed in Table~\ref{table:K3} in fact contains all elliptic curves up to isomorphism.
This is not the case for $K3$ surfaces. For instance, a general member of the family~2.2 in Table~\ref{table:K3} does not appear in the family~2.1 due to degree reasons.
\end{remark}

\section{Uniqueness of embeddings}
\label{section:linear-normality}

In this section we prove Proposition~\ref{proposition:unique-embedding}.
Let us start with a couple of facts that are well known
and can be proved similarly to their
analogs for complete intersections in the usual projective space. However,
we provide their proofs for the reader's convenience.

The proof of the following was suggested to us by A.\,Kuznetsov.

\begin{lemma}
\label{fact: splitting}
Let $X\subset\P=\P(a_0,\ldots,a_N)$ be a subvariety. 
Let $C_X^*$ be a complement of the affine cone over $X$ to its vertex, and let~\mbox{$p\colon C_X^*\to X$} be the projection.
Then one has
$$
p_*(\O_{C^*_X})=\bigoplus\limits_{m\in\ZZ} \O_X(m).
$$
\end{lemma}

\begin{proof}
Denote the polynomial ring $\CC[x_0,\ldots, x_N]$ by $R$ and let
$$
U=\big(\Spec\, R\big)\setminus \{0\}\cong\A^{N+1}\setminus\{0\}.
$$
Let
$$
Y=\Spec_\PP\left(\bigoplus\limits_{m\ge 0} \O_{\PP}(m)\right)
$$
be the relative spectrum.
Since $R\cong\Gamma\left(\bigoplus_{m\ge 0} \O_\PP(m)\right)$,
one obtains the map
$$
Y\to \Spec\, R\cong\A^{N+1}
$$
which is a weighted blow up of the origin
(with weights $a_0,\ldots,a_N$). Cutting out the origin one gets the isomorphism
$$
U\cong \Spec_\PP \left(\bigoplus_{m\in\ZZ} \O_{\PP}(m)\right).
$$
Consider the fibered product
$$
C_X^*\cong X\times_{\PP} U.
$$
Taking into account that~\mbox{$\O_\PP(m)|_X=\O_X(m)$}
by definition, we
obtain an isomorphism
$$
C_X^*\cong \Spec_X \left(\bigoplus_{m\in\ZZ} \O_{X}(m)\right),
$$
and the assertion of the lemma
follows.
\end{proof}

For a subvariety $X\subset\P(a_0,\ldots,a_N)$, we define
the Poincar\'e series
$$
P_X(t)=\sum\limits_{m\ge 0} h^0(X,\O_X(m)) t^m.
$$
The proof of the following fact was kindly shared with us by T.\,Sano.

\begin{proposition}[{see \cite[Theorem~3.4.4]{Do82}, \cite[Lemma 2.4]{PST}, \cite[Theorem~3.2.4(iii)]{Do82}}]
\label{proposition:Poincare}
Let $X\subset\P(a_0,\ldots,a_N)$ be a
weighted complete intersection
of mul\-ti\-deg\-ree~\mbox{$(d_1,\ldots,d_k)$}.
The following assertions hold.
\begin{itemize}
\item[(i)]
One has
$$
P_X(t)=\frac{\prod_{j=1}^k (1-t^{d_j})}{\prod_{i=0}^N (1-t^{a_i})}.
$$

\item[(ii)]
One has $H^i(X, \O_X(m))=0$ for all $m\in\ZZ$ and all~\mbox{$0<i<\dim X$}.
\end{itemize}
\end{proposition}

\begin{proof}
Denote
the graded polynomial ring $\CC[x_0,\ldots, x_N]$ by $R$,
denote the weighted homogeneous ideal~\mbox{$(f_1,\ldots,f_k)$} that defines $X$ by $I$, and put $S=R/I$,
so that~\mbox{$X\cong\Proj S$}.
From regularity of the sequence $f_1,\ldots,f_k$ it easily follows that
\begin{equation}\label{eq:S-Poincare}
\sum_{m\ge 0} \dim(S_m)t^m=\frac{\prod_{j=1}^k (1-t^{d_j})}{\prod_{i=0}^N (1-t^{a_i})},
\end{equation}
where $S_m$ is the $m$-th graded component of $S$.

Denote the affine cone
$$
\Spec\, S\subset \Spec\, R\cong \Aff^{N+1}
$$
over $X$ by $C_X$. By construction, one has an isomorphism
of graded algebras
\begin{equation}\label{eq:S-vs-CX}
S\cong H^0\left(C_X,\O_{C_X}\right).
\end{equation}

Following Lemma~\ref{fact: splitting} denote the complement of $C_X$ to its vertex $P$ by $C^*_X$.
Consider the local cohomology groups
$H^\bullet_P(C_X,\O_{C_X})$, see, for instance,~\cite[p.~2, Definition]{Ha67}.
We have the exact sequence
$$
\ldots\to H^i_P\left(C_X,\O_{C_X}\right)\to H^i\left(C_X,\O_{C_X}\right)\to H^i\left(C^*_X,\O_{C^*_X}\right)\to H^{i+1}_P\left(C_X,\O_{C_X}\right)\to \ldots,
$$
see~\cite[Corollary 1.9]{Ha67}.
Since $C_X$ is a complete intersection in the affine space, it is Cohen--Macaulay.
Therefore, by \cite[Proposition~3.7]{Ha67} and \cite[Theorem 3.8]{Ha67} one has
$$
H^i_P\left(C_X,\O_{C_X}\right)=0
$$
for all $i<\dim C_X=\dim X+1$.
Hence, we obtain
an isomorphism
\begin{equation}\label{eq:CX-vs-CXstar}
H^i\left(C_X,\O_{C_X}\right)\cong H^i\left(C^*_X,\O_{C^*_X}\right)
\end{equation}
for all $i<\dim X$.

Finally, denote the natural projection
$C^*_X\to X$ by $p$. Then
\begin{equation}
\label{eq:CXstar-vs-pCXstar}
H^i(C^*_X,\O_{C^*_X})\cong H^i(X, p_*\O_{C^*_X})
\end{equation}
for all $i$.
On the other hand,
by Lemma~\ref{fact: splitting} we have
\begin{equation}
\label{eq:pCXstar}
p_*(\O_{C^*_X})=\bigoplus\limits_{m\in\ZZ} \O_X(m).
\end{equation}

For $i=0$, we combine the isomorphisms~\eqref{eq:S-vs-CX}, \eqref{eq:CX-vs-CXstar},
\eqref{eq:CXstar-vs-pCXstar}, and~\eqref{eq:pCXstar} to obtain an
isomorphism
of graded algebras
\begin{equation*}
S\cong \bigoplus\limits_{m\ge 0} H^0\left(X,\O_X(m)\right).
\end{equation*}
This together with~\eqref{eq:S-Poincare} gives assertion~(i).

For $0<i<\dim X$, we combine the isomorphisms~\eqref{eq:CX-vs-CXstar},
\eqref{eq:CXstar-vs-pCXstar}, and~\eqref{eq:pCXstar} to obtain an
isomorphism
\begin{equation*}
H^i\left(C_X,\O_{C_X}\right)\cong \bigoplus\limits_{m\in\ZZ} H^i\left(X,\O_X(m)\right).
\end{equation*}
Since $C_X$ is an affine variety, we have $H^i(C_X,\O_{C_X})=0$ for all $i>0$,
see for instance~\mbox{\cite[Theorem~III.3.5]{Ha77}}.
This proves assertion~(ii).
\end{proof}

Proposition~\ref{proposition:Poincare}(i) implies
the following property
that can be considered as an analog of linear normality for usual complete intersections.

\begin{corollary}
\label{corolary:linear-normality}
Let $X\subset\P$ be a quasi-smooth well formed weighted complete intersection.
Then the restriction map
$$
H^0\big(\P,\O_{\P}(m)\big)\to H^0\big(X, \O_X(m)\big)
$$
is surjective for every $m\in\ZZ$.
\end{corollary}

\begin{proof}
The dimension of the image of the restriction map is computed by the coefficient
in the Poincar\'e series~\eqref{eq:S-Poincare}.
On the other hand, by
Proposition~\ref{proposition:Poincare}(i)
this coefficient also equals the dimension of~\mbox{$ H^0(X, \O_X(m))$}.
\end{proof}

To proceed, we will need an elementary observation.

\begin{lemma}[{see~\cite[Lemma~18.3]{IF00}}]\label{lemma:ratio}
Let $N$ and $N'$ be positive integers, and $k$ and~$k'$ be non-negative integers.
Let $a_0\le\ldots\le a_N$, $a_0'\le\ldots\le a_{N'}'$,
$d_1\le\ldots\le d_k$, and~\mbox{$d_1'\le\ldots\le d_{k'}'$} be positive integers.
Suppose that
\begin{equation}\label{eq:ratio}
\frac{\prod_{j=1}^k (1-t^{d_j})}{\prod_{i=0}^N (1-t^{a_i})}=
\frac{\prod_{j'=1}^{k'} (1-t^{d'_{j'}})}{\prod_{i'=0}^{N'} (1-t^{a'_{i'}})}
\end{equation}
as rational functions in the variable $t$.
Suppose that $a_i\neq d_j$ for all $i$ and $j$, and
$a_{i'}'\neq d_{j'}'$ for all $i'$ and $j'$.
Then $N=N'$, $k=k'$, $a_i=a'_i$ for every $0\le i\le N$, and $d_j=d'_j$ for every~\mbox{$1\le j\le k$}.
\end{lemma}

\begin{proof}
Note that numerators and denominators of the rational functions in the left and the right hand sides of~\eqref{eq:ratio}
may have common divisors (for instance, if some $d_j$ is divisible by some $a_i$, or another way around). To prove the
assertion, we will keep track of the numbers that are roots of either the numerator or the denominator, but not both of them.

Observe that equality~\eqref{eq:ratio} is equivalent to the equality  obtained
from~\eqref{eq:ratio} by interchanging the collections $\{a_i\}$, $\{a_j'\}$ with $\{d_s\}$, $\{d_r'\}$, respectively.
Thus we may assume that~$d_k$ is the maximal number among $a_N$, $a_{N'}'$, $d_k$, and $d_{k'}'$.
By assumption we know that~\mbox{$a_N<d_k$}. Let $\zeta$ be a primitive $d_k$-th root of unity.
Then $\zeta$ is a root of the numerator of the left hand side of~\eqref{eq:ratio}
but not the root of its denominator. Hence $\zeta$ is a root of the numerator~$\nu(t)$ of
the right hand side of~\eqref{eq:ratio} as well.
Since~\mbox{$d_{j'}'\le d_k$} for all $j'$, we see that~$\nu(t)$
is divisible by $1-t^{d_k}$. Cancelling the factor~\mbox{$1-t^{d_k}$} from~\eqref{eq:ratio}, we complete the proof
of the lemma by induction.
\end{proof}

Now we prove the main result of this section.

\begin{proof}[Proof of Proposition~\ref{proposition:unique-embedding}]
By Lemma~\ref{lemma:Okada},
the group $\Cl(X)$ is generated by the class of the line bundle
$\O_X(1)$, while the group $\Cl(X')$ is generated by the class of the line
bundle~$\O_{X'}(1)$.
Therefore, we see from Proposition~\ref{proposition:Poincare}(i) that
$$
\frac{\prod_{j=1}^k (1-t^{d_j})}{\prod_{i=0}^N (1-t^{a_i})}=P_X(t)=P_{X'}(t)=
\frac{\prod_{j'=1}^{k'} (1-t^{d'_{j'}})}{\prod_{i'=0}^{N'} (1-t^{a'_{i'}})}.
$$
Since neither $X$ nor $X'$ is an intersection with a linear cone,
the required assertion follows from Lemma~\ref{lemma:ratio}.
\end{proof}

\begin{remark}
The assertion of Proposition~\ref{proposition:unique-embedding} also
holds for smooth Fano weighted complete intersections of dimension~$2$.
This follows from their explicit classification, see Lemma~\ref{lemma:small-dimension}.
Note however that the assertion fails in dimension $1$. Indeed,
a conic in $\P^2$ is isomorphic to~$\P^1$ (which can be considered
as a complete intersection of codimension~$0$ in itself).
\end{remark}

\begin{remark}
We point out that the
assumption that $X$ is Fano is essential for the validity of
Proposition~\ref{proposition:unique-embedding} in dimension $2$. For instance,
there exist smooth quartics in~$\P^3$ that also have a structure of a double cover
of $\P^2$ branched in a sextic curve, see e.g.~\mbox{\cite[Proof of Theorem~4]{MM63}}.
Similarly, by Remark~\ref{remark:families} the assertion of Proposition~\ref{proposition:unique-embedding}
fails for elliptic curves.
\end{remark}

We do not know if the assertion of Proposition~\ref{proposition:unique-embedding} holds
in dimension $2$ in the case when~$X$ and $X'$ are quasi-smooth del Pezzo surfaces.
We point out that the
assumption that~$X$ alone is quasi-smooth is not enough for this.
Indeed, the weighted projective plane~\mbox{$\P(1,1,2)$}
(which can be considered as a quasi-smooth well formed weighted complete
intersection of codimension~$0$ in itself)
can be embedded as a quadratic cone into~$\P^3$ (which is not quasi-smooth).

In the proof of Theorem~\ref{theorem:automorphisms}, we will need a classification of
weighted complete intersections of large Fano index.

\begin{proposition}[{cf.~\cite[Theorem~2.7]{PSh18}}]
\label{proposition:low-coindex}
Let $X\subset \PP$ be a smooth well formed Fano weighted complete intersection of  dimension $n\ge 2$. Then
\begin{itemize}
\item[(i)] one has $i_X\le n+1$;

\item[(ii)] if $i_X=n+1$, then $X\cong\P^n$;

\item[(iii)] if $i_X=n$, then $X$ is isomorphic to a quadric in~$\P^{n+1}$;

\item[(iv)] if $i_X=n-1$ and $n\ge 3$, then $X$ is isomorphic to a hypersurface
of degree $6$ in $\P=\P(1^n, 2,3)$, or to
a hypersurface of degree $4$ in $\P=\P(1^{n+1}, 2)$, or to
a cubic hypersurface in $\P=\P^{n+1}$, or to an intersection
of two quadrics in $\P=\P^{n+2}$.
\end{itemize}
\end{proposition}

\begin{proof}
Recall that $i_X$ equals the Fano
index of $X$ by Corollary~\ref{corollary:index}.
By~\mbox{\cite[Corollary 3.1.15]{IP99}}, we know that $i_X\le n+1$; if $i_X=n+1$,
then $X$ is isomorphic to
$\P^n$; and if $i_X=n$, then $X$ is isomorphic to a quadric in $\P^{n+1}$. This proves
assertions~(i), (ii), and~(iii).

Now suppose that $i_X=n-1$ and $n\ge 3$.
Note that $\Pic(X)\cong\mathbb{Z}$ by Lemma~\ref{lemma:Okada}.
Thus it follows from the classification of smooth Fano varieties of Fano index $n-1$
(see~\cite{Fu80-84} or~\mbox{\cite[Theorem~3.2.5]{IP99}})
that $X$ is isomorphic either to one of the weighted complete intersections
listed in assertion~(iv), or to a linear section
of the Grassmannian~\mbox{$\mathrm{Gr}(2,5)$} in its Pl\"ucker embedding.

It remains to show that a weighted complete intersection $X$ cannot be isomorphic to
a linear section of the Grassmannian $\mathrm{Gr}(2,5)$. Suppose that $X$ is isomorphic
to such a variety. Then $n\le 6$.
If~\mbox{$n=3$}, then the Fano index of $X$ equals $2$ and its anticanonical
degree is equal to~$40$.
If~\mbox{$n=4$}, then the Fano index of $X$ equals $3$ and its anticanonical
degree is equal to~$405$.
In both cases we see from Remark~\ref{remark:cone} that there exists a
smooth weighted complete intersection $X$ with the same $n$ and $i_X$ that
is not an intersection with a linear cone.
This is impossible by a classification of
smooth Fano weighted complete intersections of dimensions~$3$ and~$4$,
see~\cite[Table 2]{PSh18} (cf.~\cite[\S12]{IP99})
and~\cite[Table~1]{PrzyalkowskiShramov-Weighted}, respectively.

Therefore, we see that $5\le n\le 6$. Recall that
$$
H^4\big(\mathrm{Gr}(2,5),\mathbb{Z}\big)\cong\mathbb{Z}^2.
$$
Thus, Lefschetz hyperplane section theorem implies that
$$
H^4(X,\mathbb{Z})\cong\mathbb{Z}^2.
$$
On the other hand, since $X$ is a
weighted complete intersection of dimension greater than~$4$,
one has $H^4(X,\mathbb{Z})\cong\mathbb{Z}$ by the Lefschetz-type theorem for complete intersections
in toric varieties, see~\cite[Proposition~1.4]{Ma99}.
The obtained contradiction completes the proof of assertion~(iv).
\end{proof}

Proposition~\ref{proposition:unique-embedding} allows us to prove more precise
classification results
concerning Fano weighted complete intersections (which we will not use directly
in our further proofs).

\begin{corollary}
\label{corollary:low-coindex}
Let $X\subset \PP$ be a smooth well formed Fano weighted complete intersection of  dimension $n\ge 2$
that is not an intersection with a linear cone. Then
\begin{itemize}
\item[(i)] if $i_X=n+1$, then $X=\P=\P^n$;

\item[(ii)] if $i_X=n$, then $X$ is a quadric in $\P=\P^{n+1}$;

\item[(iii)] if $i_X=n-1$, then $X$ is either a hypersurface
of degree $6$ in $\P=\P(1^n, 2,3)$, or
a hypersurface of degree $4$ in $\P=\P(1^{n+1}, 2)$, or
a cubic hypersurface in $\P=\P^{n+1}$, or an intersection
of two quadrics in $\P=\P^{n+2}$.
\end{itemize}
\end{corollary}
\begin{proof}
Assertions (i) and (ii) follow from
assertions (ii) and (iii) of Proposition~\ref{proposition:low-coindex}, respectively,
applied together with Proposition~\ref{proposition:unique-embedding}.
If $n=2$, assertion (iii) follows from Lemma~\ref{lemma:small-dimension}.
If $n\ge 3$, assertion (iii) follows from
Proposition~\ref{proposition:low-coindex}(iv) and
Proposition~\ref{proposition:unique-embedding}.
\end{proof}

\begin{remark}
An alternative way to prove
Corollary~\ref{corollary:low-coindex} (which in turn can be used
to deduce Proposition~\ref{proposition:low-coindex})
is by induction on dimension using the classification of smooth well formed Fano weighted
complete intersections of low dimension
(say, one provided by Lemma~\ref{lemma:small-dimension}) together
with~\mbox{\cite[Theorem~1.2]{PST}}.
\end{remark}

\section{Automorphisms}
\label{section:automorphisms}

In this section we prove Theorem~\ref{theorem:automorphisms}.

Let $\P=\P(a_0,\ldots,a_N)$ be a weighted projective space.
For any subvariety $X\subset\P$, we denote by
$\Aut(\P;X)$ the stabilizer of $X$ in $\Aut(\P)$. We denote by
$\Aut_{\P}(X)$
the image of~\mbox{$\Aut(\P;X)$} under the restriction map
to $\Aut(X)$. In other words, the group $\Aut_{\P}(X)$
consists of automorphisms of $X$ induced by automorphisms of $\P$.

We start with a general result that is well known to experts
(see for instance~\mbox{\cite[Lemma~3.1.2]{KPS18}})
and that was pointed out to us by A.\,Massarenti.

\begin{lemma}\label{lemma:LAG}
Let $X$ be a normal variety, let $A$ be a very ample Weil divisor
on $X$, and let $[A]$ be the class of $A$ in $\Cl(X)$. Denote by
$\Aut(X; [A])$ the stabilizer of $[A]$ in $\Aut(X)$.
Then $\Aut(X; [A])$ is a linear algebraic group.
\end{lemma}

\begin{corollary}\label{corollary:LAG}
Let $X$ be a quasi-smooth well formed weighted complete intersection
of dimension $n$.
Suppose that either $n\ge 3$, or $K_X\neq 0$.
Then $\Aut(X)$ is a linear algebraic group.
\end{corollary}
\begin{proof}
Note that the divisor class $K_X$ is $\Aut(X)$-invariant.
Moreover, if $K_X\neq 0$, then either~$K_X$ or $-K_X$ is ample
by Theorem~\ref{theorem:adjunction}.
On the other hand, if $n\ge 3$, then $\Cl(X)\cong\mathbb{Z}$
by Lemma~\ref{lemma:Okada}, so that an ample generator of $\Cl(X)$
is $\Aut(X)$-invariant. In both cases we see that $\Aut(X)$ preserves
some ample (and thus also some very ample) divisor class on $X$.
Hence $\Aut(X)$ is a linear algebraic group by Lemma~\ref{lemma:LAG}.
\end{proof}

The following lemma will not be used in the proof of
Theorem~\ref{theorem:automorphisms}, but will allow us to prove its (weaker) analog that applies to a slightly wider class of smooth weighted complete
intersections, see Corollary~\ref{corollary:Flenner}(ii) below.

\begin{lemma}\label{lemma:non-ruled}
Let $X$ be a subvariety of $\P$. Then $\Aut_{\P}(X)$ is a linear algebraic group.
\end{lemma}

\begin{proof}
The group $\Aut(\P)$ is obviously a linear algebraic group.
The stabilizer $\Aut(\P;X)$ of $X$ in $\Aut(\P)$
and the kernel $\Aut(\P;X)_{\mathrm{id}}$
of its action on $X$ are cut out in $\Aut(\P)$ by
algebraic equations, so that they are linear algebraic groups.
The group~\mbox{$\Aut(\P;X)_{\mathrm{id}}$} is a normal subgroup of~\mbox{$\Aut(\P;X)$}.
Therefore, the
group
$$
\Aut_{\P}(X)\cong\Aut(\P;X)/\Aut(\P;X)_{\mathrm{id}}
$$
is a linear algebraic group as well, see \cite[Theorem~6.8]{Borel}.
\end{proof}

\begin{corollary}\label{corollary:non-ruled}
Let $X$ be a smooth irreducible subvariety of $\P$. Suppose that $K_X$ is numerically effective.
Then the group~\mbox{$\Aut_{\P}(X)$} is finite.
\end{corollary}

\begin{proof}
By Lemma~\ref{lemma:non-ruled}, the group $\Aut_{\P}(X)$ is a linear algebraic group.
Therefore, if~\mbox{$\Aut_{\P}(X)$} is infinite, then it contains a subgroup
isomorphic either to $\CC^{\times}$ or to $\CC^+$,
which implies that~$X$ is covered by rational curves.
On the other hand, since $K_X$ is numerically effective,
$X$ can't be covered by rational
curves, see~\cite[Theorem 1]{MM86}.
\end{proof}

The main tool we use in the proof of Theorem~\ref{theorem:automorphisms}
is the following result from~\cite{Fle81}.

\begin{theorem}[{see \cite[Satz 8.11(c)]{Fle81}}]
\label{theorem:Flenner}
Let $X$ be a smooth weighted complete intersection of dimension $n\ge 2$.
Then
$$
H^n\big(X, \Omega^1_X\otimes\mathcal{O}_X(-i)\big)=0
$$
for every integer $i\le n-2$.
\end{theorem}

\begin{remark}
Actually, the assertion of \cite[Satz 8.11(c)]{Fle81} gives more vanishing results
and holds under the weaker assumption that $X$ is quasi-smooth. However, we do not want to go into
details with the definition of the sheaf $\Omega_X^1$ here, and in any case we will need smoothness of
$X$ on the next step.
\end{remark}

Theorem~\ref{theorem:Flenner} allows us to prove finiteness of various automorphism groups.

\begin{corollary}
\label{corollary:Flenner}
Let $X$ be a smooth well formed weighted complete intersection
of dimension~\mbox{$n\ge 2$}.
Suppose that $i_X\le n-2$.
The following assertions hold.
\begin{itemize}
\item[(i)] One has $H^0(X, T_X)=0$.

\item[(ii)] The group $\Aut_{\P}(X)$ is finite.

\item[(iii)] If either $\dim X\ge 3$ or $K_X\neq 0$,
then the group $\Aut(X)$ is finite.
\end{itemize}
\end{corollary}

\begin{proof}
By Theorem~\ref{theorem:Flenner}
we have
$$
H^n\big(X, \Omega^1_X\otimes\mathcal{O}_X(-i_X)\big)=0.
$$
Recall that $\omega_X=\mathcal{O}_X(-i_X)$ by Theorem~\ref{theorem:adjunction}.
Thus assertion~(i) follows from Serre duality.
Assertion~(ii) follows from assertion~(i), because
$\Aut_{\P}(X)$ is a linear algebraic group
by Lemma~\ref{lemma:non-ruled}.
Similarly, assertion~(iii) follows from assertion~(i), because
the automorphism group of any variety
subject to the above assumptions
is a linear algebraic group by Corollary~\ref{corollary:LAG}.
\end{proof}

Recall that the smooth Fano threefold $V_5$ of Fano index $\dim V_5-1=2$
that is defined as an intersection of the Grassmannian
$\mathrm{Gr}(2,5)\subset\P^9$
in its Pl\"ucker embedding with a linear section of codimension
$3$ has infinite automorphism group $\Aut(V_5)\cong\mathrm{PGL}_2(\CC)$,
see~\mbox{\cite[Proposition~4.4]{Mukai-CurvesK3Fano}}
or \cite[Proposition~7.1.10]{CheltsovShramov}.
The next lemma shows that such a situation is impossible
for smooth weighted complete intersections.

\begin{lemma}
\label{lemma:low-coindex}
Let $X\subset \PP$ be a smooth well formed weighted complete intersection of
dimension~\mbox{$n\ge 2$}. Suppose that $i_X=n-1$. Then
the group $\Aut(X)$ is finite.
\end{lemma}

\begin{proof}
If $n=2$, the assertion follows from Remark~\ref{remark:dP} and
the properties of automorphism groups of smooth del Pezzo
surfaces, see for instance \cite[Corollary~8.2.40]{Dolgachev-CAG}.
Thus, we assume that $n\ge 3$ and
use the classification provided by Proposition~\ref{proposition:low-coindex}(iv).
If $X$ is isomorphic to an intersection
of two quadrics in~\mbox{$\P=\P^{n+2}$} or to a cubic hypersurface in~\mbox{$\P=\P^{n+1}$},
then the assertion follows from
Theorem~\ref{theorem:Benoist} (in the latter case one can also
use Theorem~\ref{theorem:MM}).

Now suppose that $X$ is isomorphic either to a
 hypersurface of degree $4$ in~\mbox{$\P=\P(1^{n+1}, 2)$}, or
to a hypersurface
of degree $6$ in $\P=\P(1^n, 2,3)$.
The argument in these cases is similar to that in the proof of \cite[Lemma~4.4.1]{KPS18}.
Denote by~$H$ the ample divisor such that~\mbox{$-K_X\sim (n-1)H$}.
Then there exists
an $\Aut(X)$-equivariant double cover~\mbox{$\phi\colon X\to Y$},
where in the former case $Y\cong\P^n$ and $\phi$ is given by the linear system
$|H|$, while in the latter case~\mbox{$Y\cong\P(1^{n},2)$} and $\phi$ is given by the linear system
$|2H|$. Let $H'$ be the ample Weil divisor generating the group
$\Cl(Y)\cong\ZZ$, and let
$B\subset Y$ be the branch divisor of $\phi$. In the former
case one has $B\sim 4H'$, and in the latter case one has $B\sim 6H'$.
Note that in the latter case~$H'$ is not Cartier, but $2H'$ is;
note also that in this case $\phi$ is branched over the singular point
of $Y$ as well.
In both cases $B$ is smooth. Furthermore, it follows from adjunction
formula that either $K_B$ is ample, or $K_B\sim 0$, or $B$ is a
(smooth well formed) Fano weighted hypersurface of dimension $n-1\ge 3$ and
Fano index $i_B\le n-3$.

Since the double cover $\phi$ is $\Aut(X)$-equivariant,
we see that the quotient of the group~\mbox{$\Aut(X)$} by its normal subgroup
of order $2$ generated by the Galois involution of~$\phi$
is isomorphic to a subgroup of the stabilizer $\Aut(Y;B)$ of
$B$ in $\Aut(Y)$.
Since $B$ is not contained in any divisor linearly equivalent to
the very ample divisor $2H'$, we conclude that~\mbox{$\Aut(Y;B)$}
acts faithfully on $B$, see for
instance~\cite[Lemma~2.1]{CPS19}.
Hence
$$
\Aut(Y;B)\cong\Aut_Y(B).
$$
On the other hand, the group $\Aut_Y(B)$ is finite by
Corollary~\ref{corollary:Flenner}(ii); alternatively, one can apply
Corollaries~\ref{corollary:non-ruled} and~\ref{corollary:Flenner}(iii).
This means that the group $\Aut(X)$ is finite as well.
\end{proof}

Now we prove our main results.

\begin{proof}[Proof of Theorem~\ref{theorem:automorphisms}]
First suppose that $n=1$. We may assume
that $K_X$ is ample. In this case the finiteness
of $\Aut(X)$ is well-known, see for
instance~\mbox{\cite[Exercise~IV.5.2]{Ha77}}.

Now suppose that $n\ge 2$.
If $i_X\le n-2$, then the group $\Aut(X)$ is finite by
Corollary~\ref{corollary:Flenner}(iii).
If~\mbox{$i_X=n-1$}, then
the group $\Aut(X)$ is finite by Lemma~\ref{lemma:low-coindex}.
Finally, if~\mbox{$i_X\ge n$}, then we know from
Proposition~\ref{proposition:low-coindex}
that $X$ is isomorphic either to~$\P^n$ or to
a quadric hypersurface in~$\P^{n+1}$.
\end{proof}

Corollary~\ref{corollary:automorphisms} immediately follows
from Theorem~\ref{theorem:automorphisms} and
Proposition~\ref{proposition:unique-embedding}.

\end{document}